\documentclass[11pt,a4paper]{amsart}
\usepackage[
  a4paper,
  left=3.3cm,
  right=3.3cm,
  top=3.3cm,
  bottom=3.5cm
]{geometry}

\usepackage{amsmath,amssymb,amsthm}
\usepackage{graphicx}
\usepackage[hidelinks]{hyperref}

\newtheorem{theorem}{Theorem}[section]
\newtheorem{proposition}[theorem]{Proposition}
\newtheorem{lemma}[theorem]{Lemma}
\newtheorem{corollary}[theorem]{Corollary}
\theoremstyle{remark}

\newcommand{\N}{\mathbb N}

\title[Critical tails in the anisotropic Bak--Sneppen model]
{Critical avalanche tails in the one-dimensional anisotropic Bak--Sneppen model}


\author{Olivier Hénard}
\address{Institut de mathématique d'Orsay, Université Paris-Saclay,
91405 Orsay, France}
\email{olivier.henard@universite-paris-saclay.fr}

\hypersetup{
  pdftitle={Critical avalanche tails in the one-dimensional anisotropic Bak--Sneppen model},
  pdfauthor={Olivier Hénard}
}

\begin{document}

\begin{abstract}
\emergencystretch=2em
We study avalanche ranges in the anisotropic Bak--Sneppen model on
\(\mathbb Z\). In exponential coordinates, the right endpoint \(X_t\) of
an avalanche started at the origin has a nonlinear jump evolution whose tails satisfy a
closed triangular system. Write \(m(t)=\mathbb E[X_t]\) and let \(t_c\)
be the threshold for finite \(m(t)\). Below \(t_c\), the endpoint has
exponential moments and \(m(t)\sim2/(t_c-t)\). At criticality it is finite
almost surely, but
\[
        \mathbb P(X_{t_c}\geq k)\geq ck^{-1/2},
        \qquad
        \mathbb E[X_{t_c}^{1/3}]=\infty.
\]

We propose two approaches for the critical estimates. First, a simple
compensation argument, based on an extremal inequality for decreasing
probability masses, yields a \(k^{-2/3}\) lower bound. This already proves
equality of the susceptibility and infinite-avalanche thresholds. We then
exploit the additional fact that the exact evolution preserves geometric
mixtures for the tails, which sharpens the bound to \(k^{-1/2}\). 
Finally, by freezing the grafting law at criticality, we prove that
the avalanche is infinite with positive probability for every
\(t>t_c\).
\end{abstract}

\maketitle

\section{Introduction and main result}
\label{sec:introduction}

The Bak--Sneppen model \cite{BakSneppen1993} is a classical example of self-organized criticality
driven by extremal dynamics. Each site carries a fitness; at each update,
the site of minimum fitness and a prescribed set of its neighbours receive
fresh independent values. We consider the one-dimensional anisotropic, or
one-sided, rule, which refreshes the minimum site and its right neighbour.
For a fixed fitness threshold, an avalanche is the sequence of updates
triggered from a prescribed site and continued until all fitnesses are again
above that threshold. In the one-sided model, an avalanche started at the
origin has range \(\{0,\ldots,X_t\}\), and its right endpoint admits a closed
scalar evolution as the fitness level \(t\) varies. It is this closure, rather
than the interval shape of the range alone, that is special to the one-sided
model; Figure~\ref{fig:avalanches} illustrates this evolution.

\begin{figure}[t]
\centering
\includegraphics[width=\textwidth]{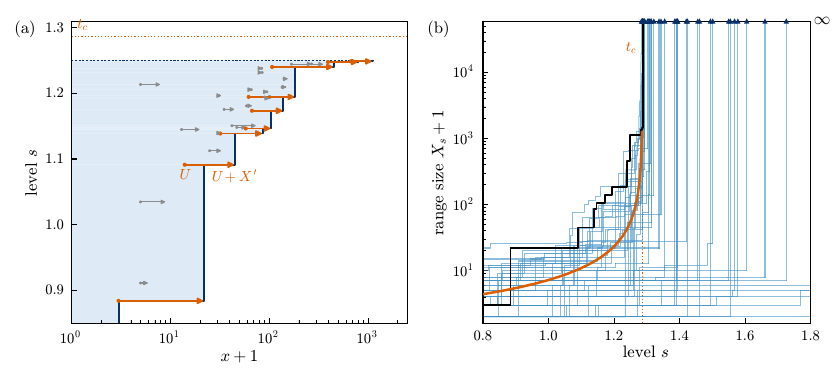}
\caption{Avalanche growth in the one-sided Bak--Sneppen model
(exponential coordinates; dotted lines: $t_c\approx1.2864$).
(a)~One avalanche in the level--space plane, up to level $1.25$
(dashed line). The shaded region is $\{(x,s):x\le X_s\}$;
arrows indicate grafts $U\mapsto U+X'$, in orange when the range grows
and grey otherwise. This is the largest of $3000$ sampled avalanches (most avalanches are still small at criticality).
(b)~Range size $X_s+1$ for $50$ independent avalanches (blue)
and for the avalanche in (a), continued (black). 
The orange curve is the mean $m(s)+1$.}
\label{fig:avalanches}
\end{figure}

We work in exponential coordinates: the order-preserving transformation
\(u\mapsto-\log(1-u)\) turns uniform refresh variables into independent
\(\operatorname{Exp}(1)\) variables. Throughout, \(t\geq0\) denotes a
fitness level, not the discrete time of the dynamics, and
\(\mathbb N=\{1,2,\ldots\}\). More precisely, a \(t\)-avalanche starts with
all fitnesses above \(t\), forces an update of sites \(0\) and \(1\), and then
runs while some fitness is at most \(t\). Its endpoint
\(X_t\in\mathbb N\cup\{\infty\}\) is infinite exactly when it does not
terminate. The number of sites in its range is \(X_t+1\), and the avalanche
family can be coupled increasingly in \(t\).

Set \(q_k(t)=\mathbb P(X_t\geq k)\) and
\[
        m(t):=\mathbb E[X_t]=\sum_{k\geq1}q_k(t).
\]
The susceptibility threshold is
\begin{equation}
\label{eq:critical-level-intro}
        t_c:=\sup\{t\geq0:m(t)<\infty\}.
\end{equation}
Our main result describes the subcritical, critical and postcritical laws.

\begin{theorem}[Quantitative avalanche phase transition]
\label{thm:main}
One has
\[
        1\leq t_c\leq2\log2,
        \qquad
        m(t)\sim\frac{2}{t_c-t}
        \quad\text{as }t\uparrow t_c.
\]
Moreover the following hold.

\begin{enumerate}
\item[\rm (i)]
For every \(t<t_c\), there exists \(\lambda_t>0\) such that
\[
        \mathbb E\!\left[e^{\lambda_t X_t}\right]<\infty.
\]

\item[\rm (ii)]
At the critical level,
\[
        \mathbb P(X_{t_c}<\infty)=1,
        \qquad
        \mathbb E[X_{t_c}^{1/3}]=\infty.
\]
and there exists \(c>0\) such that
\[
        \mathbb P(X_{t_c}\geq k)
        \geq ck^{-1/2},
        \qquad k\geq1.
\]

\item[\rm (iii)]
For every \(t>t_c\),
\[
        \mathbb P(X_t=\infty)>0.
\]
Moreover, there exist constants \(c_*>0\) and \(s_0>0\) such that
\[
        \mathbb P(X_{t_c+s}=\infty)\geq c_*s,
        \qquad 0<s\leq s_0.
\]
\end{enumerate}

Consequently,
\[
        t_c
        =
        \inf\{t\geq0:\mathbb P(X_t=\infty)>0\}.
\]
\end{theorem}

The pointwise exponent \(1/2\) is not expected to be sharp. Numerical
integration of the exact spatial-size master equation in \cite{MaslovDeLosRiosMarsiliZhang1998}
gives a critical probability-mass exponent \(\tau_r=1.299(3)\) for the
range size \(X_t+1\), corresponding formally to a tail exponent
\(\tau_r-1\simeq0.299\). Our fractional-moment estimate implies that, if
the logarithmic critical-tail exponent exists, then it is at most \(1/3\);
we do not prove that existence.

\smallskip\noindent
\emph{Exact evolution and earlier work.}
At a finite endpoint \(x\), proposed updates occur at rate \(x+1\), with
new endpoint
\begin{equation}
\label{eq:intro-range-evolution}
        \max\{x,U_x+X_t'\},
\end{equation}
where \(U_x\) is uniform on \(\{0,\ldots,x\}\), and \(X_t'\) is an
independent copy of \(X_t\), independent of \(U_x\). Consequently
\(q_1\equiv1\) and, for \(k\geq2\),
\begin{equation}
\label{eq:intro-tail-system}
        q_k'
        =\sum_{u=1}^{k-1}q_u(q_{k-u}-q_k)+q_k(1-q_k).
\end{equation}
The equation at level \(k\) involves only \(q_1,\ldots,q_k\).

Exact avalanche equations go back to Maslov \cite{Maslov1996}.
Maslov, De Los Rios, Marsili and Zhang derived the triangular master
equation for the anisotropic spatial-size distribution
\cite{MaslovDeLosRiosMarsiliZhang1998}; \eqref{eq:intro-tail-system} is
its equivalent tail formulation, with range size \(X_t+1\). The graphical
avalanche construction was developed rigorously by Meester and Znamenski
\cite{MeesterZnamenski2003,MeesterZnamenski2004}. Other rigorous
approaches include branching and percolation bounds on avalanche critical
values \cite{GillettMeesterNuyens2006}, maximal avalanches
\cite{GillettMeesterVanderWal2006}, rank-driven variants
\cite{GrinfeldKnightWade2011,GrinfeldKnightWade2012}, and, more recently,
extinction and survival criteria for the Bak--Sneppen model on arbitrary
graphs \cite{PopovVolkov2026}. Here we combine the exact one-sided
evolution with moment inequalities and a geometric-mixture representation
to obtain quantitative control at and near criticality.

\smallskip\noindent
\emph{A simple sharpness argument and its quantitative refinement.}
Summing the tail system gives
\begin{equation}
\label{eq:intro-mean-identity}
        m'=\frac12m^2+m-\frac12\sum_{k\geq1}q_k^2.
\end{equation}
This identity yields the upper bound on \(t_c\) and, together with a
finite-mean continuation argument, divergence of \(m(t)\) at \(t_c\).
It then gives the stated susceptibility asymptotic, a rigorous version of
the asymptotic derived in \cite{MaslovDeLosRiosMarsiliZhang1998}. The gain part of the system gives
\(t_c\geq1\) and, through a generating-function estimate, subcritical
exponential moments.

For the critical tail, write \(A(t)=\sum_{k\geq1}kq_k(t)\). We
proceed in two stages. First, the normalized tail law
\(\mathbb P(K=k)=q_k/m\) has a decreasing probability mass function.
Representing such laws as \textit{mixtures of discrete uniform distributions} gives
\[
        \mathbb E[K\vee K']\geq\frac43\mathbb E[K]-\frac13
\]
for independent copies \(K,K'\). Its insertion into the exact equation for
\(A\), followed by Paley--Zygmund, yields
\[
        A(t)\leq m(t)^{8/3},
        \qquad
        q_k(t_c)\geq ck^{-2/3}.
\]

The second layer is the structural result of
Proposition~\ref{prop:geometric-mixture}: for every \(t\geq0\), there is a
unique probability measure \(\rho_t\) on \([0,1]\) such that
\[
        q_k(t)=\int_{[0,1]}z^{k-1}\,\rho_t(dz),\qquad k\geq1.
\]
Thus \(X_t\) is a mixture of geometric laws, and an atom at \(z=1\)
represents an infinite avalanche. Below \(t_c\), the normalized tail law
is itself a \textit{geometric mixture}, and the sharper maximum inequality
\begin{equation}
\label{eq:intro-extremal}
        \mathbb E[K\vee K']
        \geq\frac32\mathbb E[K]-\frac14-\frac14\mathbb P(K=1)
\end{equation}
improves the same compensation argument to
\begin{equation}
\label{eq:intro-compensation}
        A(t)\leq m(t)^{5/2},\qquad t<t_c.
\end{equation}
Paley--Zygmund now gives the \(k^{-1/2}\) bound, while a truncated
interpolation argument proves the fractional-moment conclusion, including
the endpoint \(p=1/3\). A separate diagonal argument proves almost sure
finiteness at \(t_c\).

\smallskip\noindent
\emph{From critical tails to postcritical explosion.}
Freeze the graft law in \eqref{eq:intro-range-evolution} at its critical
law. More generally, if its tail is at least \(ck^{-\alpha}\), with
\(0<\alpha<1\), then, while the frozen process lies between \(x\) and
\(2x\), grafts of size at least \(2x\) occur at rate at least a constant
times \(x^{1-\alpha}\). The expected crossing time is therefore at most a
constant times \(x^{\alpha-1}\). Summing over dyadic scales gives explosion,
and a triangular comparison transfers it to the true postcritical law.
Quantitatively, this argument gives a lower bound proportional to
\(s^{\alpha/(1-\alpha)}\) on the probability of explosion by time \(s\).
Thus the elementary exponent \(2/3\) already proves sharpness, with a
quadratic postcritical lower bound; the geometric-mixture exponent \(1/2\)
improves that bound to a linear one.

Section~\ref{sec:avalanche-range} develops the avalanche hierarchy and
exact endpoint evolution. Section~\ref{sec:subcritical} treats subcritical
susceptibility and exponential moments. Section~\ref{sec:critical-tail}
develops the two compensation arguments and proves critical finiteness, and
Section~\ref{sec:postcritical} establishes the postcritical bounds.

\medskip
\noindent
\textit{Version note.}
This paper replaces the earlier manuscript entitled \emph{Sharpness of the
avalanche phase transition in the Bak--Sneppen model}. The isotropic and
stationary-limit results are now treated separately in the companion work
\cite{HenardZd}, which considers the nearest-neighbour model on
\(\mathbb Z^d\) by dimension-independent softer  arguments. The present
paper restricts to the stronger quantitative results that rely on the closed
one-sided endpoint evolution, and adds in particular the improved bound for the tail exponent.

\subsection*{Disclosure on the use of AI}
A detailed account of the role played by AI is given in Appendix~\ref{appendix-AI}. In view of the extent of the AI contribution, we do not intend to submit this work for publication.

\medskip
\noindent
\textit{Acknowledgements.}
The author thanks Linglong Yuan and Nicolas Fournier for helpful early
exchanges.

\section{The avalanche hierarchy and its exact endpoint evolution}
\label{sec:avalanche-range}

By one-sidedness, an avalanche started at the origin never updates a negative
site, so there is no loss in working on
\(\mathbb Z_+:=\{0,1,\ldots\}\). We use exponential coordinates. Thus every
refreshed fitness has law
\(\operatorname{Exp}(1)\), and a threshold \(t\geq0\) corresponds to the
level \(1-e^{-t}\) in the usual uniform coordinates.

Fix \(t\geq0\). Start from a configuration whose fitnesses all exceed
\(t\), force one update of the sites \(0\) and \(1\), and then, for as
long as some fitness is at most \(t\), update the site carrying the
smallest such fitness together with its right neighbour. Since all
untouched fitnesses exceed \(t\), their actual values play no role. The
law of the resulting procedure therefore does not depend on the initial
configuration above level \(t\).

Let \(\xi(t)\) be the set of sites refreshed at least once during this
\(t\)-avalanche, including the forced update, and let
\[
        X_t:=\sup \xi(t)\in\mathbb N\cup\{\infty\},
\]
with the convention that \(X_t=\infty\) when the range is unbounded.

\begin{lemma}
\label{lem:range-interval}
For every \(t\geq0\), almost surely,
\[
        \xi(t)=\{0,\ldots,X_t\}
\]
when \(X_t<\infty\), whereas \(\xi(t)=\mathbb Z_+\) when
\(X_t=\infty\). Moreover, the \(t\)-avalanche terminates if and only if
\(X_t<\infty\).
\end{lemma}

\begin{proof}
After the forced update the range is \(\{0,1\}\). If the current range
is \(\{0,\ldots,\ell\}\) and the selected site is \(i\leq\ell\), then
the refreshed pair is \(\{i,i+1\}\subseteq\{0,\ldots,\ell+1\}\).
Consequently the range remains an interval with left endpoint \(0\).
In particular, an unbounded range must equal \(\mathbb Z_+\).

It remains to rule out infinite duration with bounded range. Fix
\(L\geq1\). On the event that the range never leaves
\(\{0,\ldots,L\}\), there are at most \(L+1\) sites whose current
fitness can be at most \(t\). Consider the next \(L+1\) updates and the
event that, at each of them, both refreshed fitnesses exceed \(t\).
Its conditional probability is \(e^{-2t(L+1)}\). On this event, either
the range leaves \(\{0,\ldots,L\}\), or no new site with fitness at
most \(t\) is created and at least one such site disappears at every
update; in the latter case the avalanche terminates within \(L+1\)
steps. Let \(E\) be the event that the avalanche continues for the next
\(L+1\) updates while its range stays in \(\{0,\ldots,L\}\). Thus,
conditional on any past history,
\[
        \mathbb P(E\mid\text{past})\leq1-e^{-2t(L+1)}.
\]
Applying this estimate to successive blocks gives
\[
 \mathbb P\bigl(
   \text{the avalanche never terminates},\ X_t\leq L
 \bigr)=0.
\]
Taking the union over \(L\) completes the proof.
\end{proof}

We next couple the avalanche ranges as the level varies. To make threshold
continuation unambiguous, fix a finite upper level \(T>0\) and start with
all untouched fitnesses equal to \(T+1\). They remain inactive at every
level in \([0,T]\). Since the fixed-level avalanche law does not depend
on the initial values above its threshold, this choice leaves all the
required marginal laws unchanged.

The fixed-level endpoint laws form a measurable probability kernel in
\(s\). Indeed, using a sequence of independent pairs of refresh variables,
the event that an \(s\)-avalanche reaches a given site within its first
\(n\) updates is a Borel function of \(s\) and of finitely many refresh
variables. Taking the union over \(n\) and then integrating shows that every
endpoint tail is Borel in \(s\). The finite-state masses and the mass at
infinity are then Borel as well. This argument does not use the tail
equations derived below.

During the exploration of a fixed \(t\)-avalanche, reveal a refresh
variable exactly when it is at most \(t\), and otherwise reveal only
that it exceeds \(t\). Let \(\mathcal H_t\) be the sigma-field generated
by the update history, the identities of the refresh variables used, and
this partially revealed information. Values above \(t\) need not be
compared with one another, since they are inactive.

\begin{lemma}[Regeneration at a level]
\label{lem:level-regeneration}
Fix \(0\leq t\leq T\), and suppose that the \(t\)-avalanche terminates
with range \(I=\{0,\ldots,x\}\). Conditionally on \(\mathcal H_t\), the terminal
fitnesses on \(I\) are independent and have the form
\[
        F_y=t+E_y,\qquad y\in I,
\]
where the \(E_y\) are independent \(\operatorname{Exp}(1)\) variables.

Consequently, the next terminal fitness above \(t\) occurs after an
exponential level increment of rate \(|I|=x+1\), at a uniform site of
\(I\), independent of that increment. If this crossing occurs at a site
\(y\) and a level \(s\leq T\), the range of the grafted avalanche,
conditionally on the past, \(s\) and \(y\), has the law of an
\(s\)-avalanche translated by \(y\); it is generated using fresh refresh
variables independent of the past.
\end{lemma}

\begin{proof}
At termination, the last refresh variables assigned to the sites of \(I\)
are distinct. The exploration has revealed each only through the event
that it exceeds \(t\). Since variables above \(t\) are never compared
with one another, the transcript imposes no further restriction on their
values. Independence of the refresh variables and the lack-of-memory
property therefore give the stated conditional product law.

The minimum of the excesses is exponential with rate \(|I|\); its
location is uniform on \(I\) and independent of the minimum. At the
corresponding level \(s\leq T\), all other old terminal fitnesses exceed
\(s\), as do the untouched initial fitnesses. Updating the crossed site
and its right neighbour therefore initiates an \(s\)-avalanche.
Conditionally on \(s\) and its origin, its fresh refresh variables and
translation invariance give the asserted graft law.

If the graft terminates, the old terminal variables that it did not
refresh still have independent \(s+\operatorname{Exp}(1)\) laws by lack
of memory. The last variables on the sites it refreshed have the same
conditional laws by the preceding exploration argument. Thus regeneration
is preserved after each graft, including grafts that do not enlarge the
range.
\end{proof}

Lemma~\ref{lem:level-regeneration} is the one-sided specialization of the
hierarchical avalanche decomposition used for the nearest-neighbour model
in \cite{HenardZd}. The common mechanism is regeneration whenever a
terminal fitness is crossed; here the interval geometry makes each graft
reducible to a scalar endpoint update.

We now construct the graphical family from these fixed-level endpoint
laws. Independently for each \(y\in\mathbb Z_+\), take a marked Poisson
point process on \((0,\infty)\times(\mathbb N\cup\{\infty\})\) with
intensity
\[
        ds\,\mathbb P(X_s\in dz).
\]
Equivalently, the heights on each line form a rate-one Poisson process;
conditionally on all heights, the marks are independent, and a mark at
height \(s\) has law \(\mathcal L(X_s)\). Thus the construction couples
fixed-level laws, rather than defining them recursively.

At a point \((s,z)\) on line \(y\), draw horizontal arrows from
\((y,s)\) to \((y+j,s)\) for \(0\leq j\leq z\), or to every site on
its right when \(z=\infty\). Vertical motion is allowed upwards. From the
seed \(\{0,1\}\) at level \(0\), let \(\xi^{\mathrm{gr}}(t)\) be the
set of sites reachable at height \(t\), and let
\(X_t^{\mathrm{gr}}=\sup\xi^{\mathrm{gr}}(t)\).

For comparison, continue the original avalanche by raising its threshold
on \([0,T]\), using Lemma~\ref{lem:level-regeneration} whenever a terminal
fitness is crossed. Denote this endpoint family by
\(X_t^{\mathrm{con}}\). At any fixed target level \(t\leq T\), its
concatenated updates are exactly those of a \(t\)-avalanche: when a graft
ends, either all terminal fitnesses exceed \(t\), or the next one to be
crossed is the minimum and initiates the next graft. If infinitely many
updates are required, the endpoint is set to infinity, consistently with
Lemma~\ref{lem:range-interval}.

To justify the comparison without any non-explosion assumption, for
\(n\geq2\) set
\[
 \sigma_n^{\mathrm{gr}}:=\inf\{t\geq0:X_t^{\mathrm{gr}}\geq n\},
 \qquad
 \sigma_n^{\mathrm{con}}:=\inf\{t\geq0:X_t^{\mathrm{con}}\geq n\},
\]
with the second infimum taken over \([0,T]\), and \(\inf\varnothing=\infty\).
Before either stopping level, only the fixed lines
\(\{0,\ldots,n-1\}\) can be active, and the total proposed-update rate
is at most \(n\). There are therefore almost surely only finitely many
relevant marks on \([0,T]\) before this stopping level.
Lemma~\ref{lem:level-regeneration} identifies successively the conditional
law of the next level increment, origin and graft in the two constructions.
It follows that
\[
 \bigl(X_{t\wedge\sigma_n^{\mathrm{gr}}}^{\mathrm{gr}}\bigr)_{0\leq t\leq T}
 \stackrel{\mathrm d}{=}
 \bigl(X_{t\wedge\sigma_n^{\mathrm{con}}}^{\mathrm{con}}\bigr)_{0\leq t\leq T},
\]
including the update that crosses \(n\). Letting \(n\to\infty\) identifies
the endpoint laws on \([0,T]\), with infinity absorbing. Since \(T\) is
arbitrary, the graphical construction gives a simultaneous increasing
coupling of all fixed-level avalanche laws.

This is the one-sided specialization of the graphical representation of
Meester and Znamenski \cite[Section~4]{MeesterZnamenski2004}. We henceforth use this coupling
and again write \(\xi(t)\) and \(X_t\) for its range and endpoint; in
particular, \(X_0=1\).

The graphical construction immediately gives an exact nonlinear
evolution in the level parameter.

\begin{proposition}[Exact endpoint evolution]
\label{prop:range-evolution}
The process \((X_t)_{t\geq0}\) is increasing. At level \(t\), conditionally
on the current finite state \(X_t=x\), marks affecting the range arrive
at instantaneous rate \(x+1\). At such a mark the endpoint is replaced by
\[
        \max\{x,U_x+X_t'\},
\]
where \(U_x\) is uniform on \(\{0,\ldots,x\}\), \(X_t'\) is an
independent copy of \(X_t\), and \(U_x\) and \(X_t'\) are independent.
Null changes are allowed.

Equivalently, for every bounded function \(f\) on
\(\mathbb N\cup\{\infty\}\), the time-dependent jump operator at a finite
state \(x\) is
\[
        \mathcal L_t f(x)
        =(x+1)\,
        \mathbb E\!\left[
          f\bigl(\max\{x,U_x+X_t'\}\bigr)-f(x)
        \right].
\]
\end{proposition}

\begin{proof}
When the current range is \(\{0,\ldots,x\}\), exactly \(x+1\) vertical
lines are active. Their independent rate-one Poisson processes
superpose to a process of rate \(x+1\), and the line carrying the next
mark is uniform on the current range. Write its position as \(U_x\).
A mark at level \(t\) carries an independent \(t\)-avalanche, whose range
is \(\{0,\ldots,X_t'\}\). Translating it by \(U_x\) and taking the union
with the existing interval gives
\[
        \{0,\ldots,\max\{x,U_x+X_t'\}\}.
\]
This proves both assertions.
\end{proof}

For \(k\geq1\), set
\[
        q_k(t):=\mathbb P(X_t\geq k),
        \qquad
        q_\infty(t):=\mathbb P(X_t=\infty)
                     =\lim_{k\to\infty}q_k(t).
\]
The exact endpoint evolution closes at the level of these tail
probabilities.

\begin{proposition}[Triangular tail system]
\label{prop:tail-system}
For every \(k\geq1\), \(q_k\) is non-decreasing and continuously
differentiable. Moreover \(q_1\equiv1\), \(q_k(0)=0\) for \(k\geq2\), and
for every \(k\geq2\),
\begin{equation}
\label{eq:tail-system}
 q_k'
 =
 \sum_{u=1}^{k-1}q_u(q_{k-u}-q_k)
 +q_k(1-q_k).
\end{equation}
In particular, the marginal law of \(X_t\) is uniquely determined by
this triangular system.
\end{proposition}

\begin{proof}
Monotonicity of \(X_t\) gives monotonicity of every \(q_k\), while
\(q_1\equiv1\) and the initial conditions follow from \(X_0=1\).

Fix \(k\geq2\). While \(X_t=x<k\), Proposition~\ref{prop:range-evolution}
shows that the instantaneous rate at which an update crosses level \(k\)
is
\[
 (x+1)\mathbb P(U_x+X_t'\geq k)
 =\sum_{j=0}^{x}q_{k-j}(t).
\]
Writing
\[
        p_x(t):=\mathbb P(X_t=x)=q_x(t)-q_{x+1}(t),
\]
the compensator identity therefore gives, for \(0\leq s\leq t\),
\[
 q_k(t)-q_k(s)
 =
 \int_s^t
 \sum_{x=1}^{k-1}
 (q_x-q_{x+1})
 \sum_{j=0}^x q_{k-j}\,du,
\]
where all tail probabilities in the integrand are evaluated at \(u\).
A summation by parts yields
\[
 \sum_{x=1}^{k-1}
 (q_x-q_{x+1})
 \sum_{j=0}^x q_{k-j}
 =
 \sum_{u=1}^{k-1}q_u(q_{k-u}-q_k)+q_k(1-q_k).
\]
Hence
\begin{equation}
\label{eq:tail-system-integrated}
 q_k(t)-q_k(s)
 =
 \int_s^t
 \left[
   \sum_{u=1}^{k-1}q_u(q_{k-u}-q_k)
   +q_k(1-q_k)
 \right]du.
\end{equation}

The integrand is bounded, so \(q_k\) is absolutely continuous. Inductively
in \(k\), \(q_1,\ldots,q_{k-1}\) are continuous; since \(q_k\) is itself
continuous, the integrand in \eqref{eq:tail-system-integrated} is
continuous. Thus \(q_k\) is continuously differentiable and satisfies
\eqref{eq:tail-system} pointwise.

Finally, once \(q_1,\ldots,q_{k-1}\) are known, the equation for \(q_k\)
is a scalar locally Lipschitz differential equation with prescribed
initial condition. Uniqueness therefore follows recursively in \(k\).
\end{proof}

If \(p_r(t):=\mathbb P(X_t+1=r)\) denotes the range-size law, the two
descriptions are related by
\[
        p_r=q_{r-1}-q_r\quad(r\geq2),
        \qquad
        q_k=1-\sum_{r=2}^{k}p_r\quad(k\geq1),
\]
where an empty sum is zero. These identities remain valid above
criticality: the missing mass \(1-\sum_{r\geq2}p_r\) is precisely
\(q_\infty\). Applying this change of variables shows that
\eqref{eq:tail-system} is the tail form of the spatial-size master equation
in \cite{MaslovDeLosRiosMarsiliZhang1998}, with the endpoint shifted by one relative to the range
size. The tail formulation is particularly convenient for the moment
estimates below.

Equation~\eqref{eq:tail-system} is the only input from the graphical
avalanche construction needed in the quantitative analysis below.

\section{Subcritical susceptibility and exponential moments}
\label{sec:subcritical}

Recall that
\[
        m(t):=\mathbb E[X_t]=\sum_{k\geq1}q_k(t),
        \qquad
        t_c:=\sup\{t\geq0:m(t)<\infty\}.
\]
We first determine the evolution of the susceptibility and then strengthen
finite susceptibility to exponential integrability.

It is convenient to rewrite \eqref{eq:tail-system} in gain--loss form:
for \(k\geq2\),
\begin{equation}
\label{eq:tail-gain-loss}
 q_k'
 =
 \sum_{u=1}^{k-1}q_uq_{k-u}
 -q_k\sum_{u=2}^{k}q_u.
\end{equation}
Since \(k\mapsto q_k(t)\) is non-increasing, every term in the original
form \eqref{eq:tail-system} is non-negative. In particular,
\begin{equation}
\label{eq:q-prime-gain-bound}
        0\leq q_k'(t)
        \leq\sum_{u=1}^{k-1}q_u(t)q_{k-u}(t).
\end{equation}

We shall repeatedly sum the tail equations. The following elementary
justification will also be used in Section~\ref{sec:critical-tail}.

\begin{lemma}[Termwise differentiation]
\label{lem:tail-differentiation}
Fix \(t_0<t_c\). Then \(\sum_{k\geq1}q_k'(t)\) converges uniformly on
\([0,t_0]\), and
\[
        m'(t)=\sum_{k\geq1}q_k'(t),
        \qquad 0\leq t\leq t_0.
\]
If in addition \(\sum_{k\geq1}kq_k(t_0)<\infty\), then
\(\sum_{k\geq1}kq_k'(t)\) also converges uniformly on \([0,t_0]\), and
\[
        \frac{d}{dt}\sum_{k\geq1}kq_k(t)
        =\sum_{k\geq1}kq_k'(t).
\]
\end{lemma}

\begin{proof}
For \(0\leq t\leq t_0\), monotonicity in the level and
\eqref{eq:q-prime-gain-bound} give
\[
 q_k'(t)
 \leq
 \sum_{u=1}^{k-1}q_u(t_0)q_{k-u}(t_0).
\]
The sum of the right-hand side over \(k\geq2\) equals \(m(t_0)^2\), so
the Weierstrass test gives uniform convergence and permits termwise
differentiation.

For the weighted statement, writing
\(A(t_0):=\sum_{k\geq1}kq_k(t_0)\),
\[
 \sum_{k\geq2}k\sum_{u=1}^{k-1}q_u(t_0)q_{k-u}(t_0)
 =
 \sum_{u,v\geq1}(u+v)q_u(t_0)q_v(t_0)
 =
 2m(t_0)A(t_0)<\infty,
\]
and the same argument applies.
\end{proof}

\begin{proposition}[Susceptibility identity]
\label{prop:susceptibility}
For every \(t<t_c\),
\begin{equation}
\label{eq:mean-identity}
        m'(t)
        =
        \frac12m(t)^2+m(t)
        -\frac12\sum_{k\geq1}q_k(t)^2.
\end{equation}
Consequently,
\begin{equation}
\label{eq:riccati-lower}
        m'(t)\geq\frac12m(t)(m(t)+1),
\end{equation}
so that
\[
        1\leq t_c\leq2\log2.
\]
Moreover,
\begin{equation}
\label{eq:integrated-susceptibility}
        \int_0^t(m(s)+1)\,ds\leq2\log m(t),
        \qquad t<t_c,
\end{equation}
and, more precisely,
\begin{equation}
\label{eq:susceptibility-asymptotic}
        m(t)\sim\frac{2}{t_c-t}
        \qquad\text{as }t\uparrow t_c.
\end{equation}
\end{proposition}

\begin{proof}
We first prove \(t_c\geq1\). Set
\(m_n(t):=\sum_{k=1}^nq_k(t)\). By
\eqref{eq:q-prime-gain-bound},
\[
        m_n'(t)
        \leq
        \sum_{\substack{u,v\geq1\\u+v\leq n}}
        q_u(t)q_v(t)
        \leq m_n(t)^2.
\]
Since \(m_n(0)=1\), comparison with \(y'=y^2\), \(y(0)=1\), gives
\(m_n(t)\leq(1-t)^{-1}\) for \(t<1\). Letting \(n\to\infty\) proves
\(m(t)<\infty\) for \(t<1\).

Now fix \(t<t_c\). Lemma~\ref{lem:tail-differentiation} permits summation
of \eqref{eq:tail-gain-loss}. The gain term is \(m^2\). For the loss
term, writing \(a_k=q_k\) for \(k\geq2\),
\[
 \sum_{k\geq2}a_k\sum_{u=2}^ka_u
 =
 \frac12\left[
   \left(\sum_{k\geq2}a_k\right)^2
   +\sum_{k\geq2}a_k^2
 \right].
\]
Since \(q_1=1\), this gives \eqref{eq:mean-identity}. As
\(0\leq q_k\leq1\), one has \(\sum_kq_k^2\leq m\), and therefore
\eqref{eq:riccati-lower}.

The solution of \(y'=y(y+1)/2\) with \(y(0)=1\) blows up at
\(2\log2\), hence \(t_c\leq2\log2\). Dividing
\eqref{eq:riccati-lower} by \(m\) and integrating gives
\eqref{eq:integrated-susceptibility}.

It remains to prove divergence at \(t_c\). Since \(m\) is non-decreasing,
suppose that \(m(t)\uparrow M<\infty\) as \(t\uparrow t_c\). Monotone
convergence then gives \(m(t_c)=M\). For every \(n\), the estimate
\(m_n'\leq m_n^2\), now integrated from \(t_c\), yields
\[
        m_n(t_c+s)
        \leq
        \frac{m_n(t_c)}{1-sm_n(t_c)}
        \leq
        \frac{M}{1-sM},
        \qquad 0\leq s<M^{-1}.
\]
Letting \(n\to\infty\) shows that \(m(t_c+s)<\infty\) for
\(0<s<M^{-1}\), contradicting the definition of \(t_c\). Hence
\(m(t)\to\infty\) as \(t\uparrow t_c\).

To obtain the asymptotic, set \(S_2(t):=\sum_{k\geq1}q_k(t)^2\). Since
\(q_1=1\) and \(0\leq q_k\leq1\), we have \(1\leq S_2(t)\leq m(t)\).
Dividing \eqref{eq:mean-identity} by \(m(t)^2\) therefore gives
\[
 \frac12+\frac{1}{2m(t)}
 \leq \frac{m'(t)}{m(t)^2}
 \leq \frac12+\frac1{m(t)}-\frac{1}{2m(t)^2}.
\]
Thus \(m'(t)/m(t)^2\to1/2\) as \(t\uparrow t_c\). Writing
\(h(t):=1/m(t)\), we have \(h(t)\to0\) and \(h'(t)\to-1/2\). Hence
\[
 h(t)=-\int_t^{t_c}h'(s)\,ds\sim\frac{t_c-t}{2},
\]
which proves \eqref{eq:susceptibility-asymptotic}.
\end{proof}

The gain estimate contains more information than finite susceptibility:
strictly below \(t_c\), it propagates an exponential moment.

\begin{proposition}[Subcritical exponential moments]
\label{prop:exponential-moments}
For every \(t_0<t_c\), there exists \(\lambda_{t_0}>0\) such that
\[
        \sup_{0\leq t\leq t_0}
        \mathbb E[e^{\lambda_{t_0}X_t}]<\infty.
\]
If \(t_0>0\), one may take
\begin{equation}
\label{eq:explicit-lambda}
        \lambda_{t_0}
        =
        \log\left(1+\frac{1}{2t_0m(t_0)^4}\right).
\end{equation}
Consequently, for every \(t<t_c\) there exist \(c_t,C_t>0\) such that
\[
        \mathbb P(X_t\geq k)\leq C_te^{-c_tk},
        \qquad k\geq1,
\]
and \(X_t\) has moments of every positive order.
\end{proposition}

\begin{proof}
Fix \(t_0<t_c\) and \(a>1\), and set
\[
        \Psi_n(t):=
        \sum_{k=1}^n(a^k-1)q_k(t).
\]
By \eqref{eq:q-prime-gain-bound},
\[
 \Psi_n'
 \leq
 \sum_{u,v=1}^n(a^{u+v}-1)q_uq_v
 \leq2m\Psi_n+\Psi_n^2,
\]
where we used
\[
 a^{u+v}-1
 =
 (a^u-1)+(a^v-1)
 +(a^u-1)(a^v-1).
\]

Let
\[
        B(t):=\exp\left\{2\int_0^tm(s)\,ds\right\},
        \qquad
        \delta:=a-1.
\]
Since \(\Psi_n(0)=\delta\),
\[
        \left(\frac{\Psi_n}{B}\right)'
        \leq
        B\left(\frac{\Psi_n}{B}\right)^2,
\]
and therefore, as long as the denominator is positive,
\begin{equation}
\label{eq:Psi-bound}
        \Psi_n(t)
        \leq
        \frac{\delta B(t)}
        {1-\delta\int_0^tB(s)\,ds}.
\end{equation}
By \eqref{eq:integrated-susceptibility},
\[
        2\int_0^tm(s)\,ds
        \leq4\log m(t)-2t,
\]
so \(B(t)\leq m(t)^4\leq m(t_0)^4\) on \([0,t_0]\). If \(t_0>0\),
choose
\[
        \delta=\frac{1}{2t_0m(t_0)^4}.
\]
Then
\(\delta\int_0^{t_0}B(s)\,ds\leq1/2\), and
\eqref{eq:Psi-bound} is bounded uniformly in \(n\) and
\(t\in[0,t_0]\). Letting \(n\to\infty\) gives
\[
 \sup_{0\leq t\leq t_0}
 \sum_{k\geq1}(a^k-1)q_k(t)<\infty.
\]

For any \(W\in\N\cup\{\infty\}\), Tonelli's theorem gives, in the
extended sense,
\[
 \mathbb E[a^W]-1
 =
 (a-1)
 \sum_{k\geq1}a^{k-1}\mathbb P(W\geq k).
\]
Since
\((a-1)a^{k-1}\leq a^k-1\), the preceding bound proves
the exponential moment estimate with \(\lambda_{t_0}=\log a\).
The case \(t_0=0\) is immediate. Markov's inequality gives the
exponential tail bound, and hence all polynomial moments are finite.
\end{proof}

For later use set
\begin{equation}
\label{eq:A-definition}
        A(t):=\sum_{k\geq1}kq_k(t),
        \qquad t<t_c.
\end{equation}
Proposition~\ref{prop:exponential-moments} gives \(A(t)<\infty\). Since
\(x^2=\sum_{k=1}^x(2k-1)\), Tonelli's theorem yields
\begin{equation}
\label{eq:second-moment-tail}
        \mathbb E[X_t^2]=2A(t)-m(t)<\infty,
        \qquad t<t_c.
\end{equation}
In particular, the weighted part of
Lemma~\ref{lem:tail-differentiation} applies on every compact subinterval
of \([0,t_c)\).

\section{Critical tails}
\label{sec:critical-tail}

\subsection{Decreasing-mass compensation}

We first bound the critical tail from below using only the fact that the
normalized tail law has a decreasing probability mass function. The argument
will later be sharpened by exploiting additional structure of the exact
evolution.

\begin{lemma}[Mixtures of discrete uniform laws]
\label{lem:uniform-mixture}
Let \((p_k)_{k\geq1}\) be a non-increasing probability mass function and
set \(a_n:=n(p_n-p_{n+1})\). Then \((a_n)_{n\geq1}\) is a probability
mass function and
\[
        p_k=\sum_{n\geq k}\frac{a_n}{n},
        \qquad k\geq1.
\]
Equivalently, if \(\mathbb P(N=n)=a_n\) and, conditionally on \(N=n\),
\(K\) is uniform on \(\{1,\ldots,n\}\), then \(K\) has law
\((p_k)_{k\geq1}\).
\end{lemma}

\begin{proof}
Monotonicity and summability imply \(np_n\to0\). Indeed,
\[
        np_n
        \leq2\sum_{j=\lceil n/2\rceil}^{n}p_j
        \longrightarrow0.
\]
Hence
\[
 \sum_{n=1}^Ma_n
 =\sum_{n=1}^Mp_n-Mp_{M+1}\longrightarrow1,
\]
while
\[
        \sum_{n\geq k}\frac{a_n}{n}
        =\sum_{n\geq k}(p_n-p_{n+1})
        =p_k.
\]
\end{proof}

\begin{lemma}[Maximum under a decreasing mass function]
\label{lem:decreasing-mass}
Let \(K,K'\) be independent and identically distributed random variables
with a non-increasing probability mass function on \(\N\) and finite
mean. Then
\begin{equation}
\label{eq:maximum-inequality}
        \mathbb E[K\vee K']
        \geq
        \frac43\mathbb E[K]-\frac13.
\end{equation}
\end{lemma}

\begin{proof}
Let \(U_a,U_b\) be independent and uniform on
\(\{1,\ldots,a\}\) and \(\{1,\ldots,b\}\), with \(a\leq b\). Since
\[
 \mathbb E[U_a\wedge U_b]
 =\sum_{j=1}^a\mathbb P(U_a\geq j,U_b\geq j)
 =\frac1{ab}\sum_{j=1}^a(a-j+1)(b-j+1),
\]
and \(U_a\vee U_b=U_a+U_b-U_a\wedge U_b\), summing the quadratic
polynomial gives
\[
 \mathbb E[U_a\vee U_b]
 =
 \frac{a+b+1}{3}
 +\frac{(b-a)^2+b-1}{6b}
 \geq\frac{a+b+1}{3}.
\]
Apply Lemma~\ref{lem:uniform-mixture}, and let \(J,J'\) be independent
with law \((a_n)\). Conditionally on \((J,J')\), represent \(K,K'\) by
independent uniforms on \(\{1,\ldots,J\}\) and
\(\{1,\ldots,J'\}\). Then
\[
        \mathbb E[K\vee K']
        \geq
        \frac{2\mathbb E[J]+1}{3}
        =
        \frac43\mathbb E[K]-\frac13,
\]
because
\(\mathbb E[K]=(\mathbb E[J]+1)/2\).
\end{proof}

\begin{lemma}[Weighted-tail identity]
\label{lem:weighted-tail-identity}
For \(t<t_c\), let \(K,K'\) be independent with
\[
        \mathbb P(K=k)=\frac{q_k(t)}{m(t)},\qquad k\geq1.
\]
Then \(\mathbb E[K]=A(t)/m(t)\) and
\begin{equation}
\label{eq:A-prime-maximum}
 A'
 =
 2mA+A-\frac12m^2\mathbb E[K\vee K']
 -\frac12\sum_{k\geq1}kq_k^2.
\end{equation}
\end{lemma}

\begin{proof}
Fix \(t<t_c\) and suppress \(t\) from the notation. By
Proposition~\ref{prop:exponential-moments} and
Lemma~\ref{lem:tail-differentiation}, the weighted tail equation may be
summed term by term. From \eqref{eq:tail-gain-loss},
\[
        A'=2mA-\sum_{k\geq2}\sum_{u=2}^k kq_uq_k.
\]
Clearly \(\mathbb E[K]=A/m\). Moreover,
\[
\begin{aligned}
 m^2\mathbb E[K\vee K']
 &=2\sum_{1\leq u<k}kq_uq_k+\sum_{k\geq1}kq_k^2,\\
 \sum_{k\geq2}\sum_{u=2}^k kq_uq_k
 &=\sum_{1\leq u<k}kq_uq_k+\sum_{k\geq1}kq_k^2-A.
\end{aligned}
\]
Substitution gives \eqref{eq:A-prime-maximum}.
\end{proof}

\begin{proposition}[First compensation estimate]
\label{prop:first-moment-comparison}
For every \(t<t_c\),
\begin{equation}
\label{eq:first-moment-comparison}
        A(t)\leq m(t)^{8/3}.
\end{equation}
\end{proposition}

\begin{proof}
With the notation of Lemma~\ref{lem:weighted-tail-identity}, the mass function
\(q_k/m\) is non-increasing. Lemma~\ref{lem:decreasing-mass} and
\eqref{eq:A-prime-maximum} give
\[
 A'\leq\frac43mA+A+\frac16m^2
 -\frac12\sum_{k\geq1}kq_k^2.
\]
By Jensen's inequality and \eqref{eq:second-moment-tail},
\(m^2\leq2A-m\), hence
\[
 A'\leq\frac43mA+\frac43A-\frac16m
 -\frac12\sum_{k\geq1}kq_k^2
 \leq\frac43(m+1)A.
\]
Together with \eqref{eq:riccati-lower},
\[
        \frac{A'}A\leq\frac43(m+1)
        \leq\frac83\frac{m'}m.
\]
Since \(A(0)=m(0)=1\), integration yields
\eqref{eq:first-moment-comparison}.
\end{proof}

\begin{corollary}[A first lower bound on the critical tail]
\label{prop:first-critical-lower-tail}
There exists \(c>0\) such that
\begin{equation}
\label{eq:first-critical-lower-tail}
        q_k(t_c)\geq ck^{-2/3},\qquad k\geq1.
\end{equation}
\end{corollary}

\begin{proof}
By \eqref{eq:second-moment-tail} and
Proposition~\ref{prop:first-moment-comparison},
\(\mathbb E[X_t^2]\leq2m(t)^{8/3}\) for \(t<t_c\). Paley--Zygmund gives
\[
 \mathbb P\!\left(X_t\geq\frac12m(t)\right)
 \geq\frac14\frac{m(t)^2}{\mathbb E[X_t^2]}
 \geq\frac18m(t)^{-2/3}.
\]
For every \(k\geq1\), choose \(t_k<t_c\) with \(m(t_k)=2k\). Monotonicity
in the level then yields
\[
        q_k(t_c)\geq q_k(t_k)
        \geq\frac18(2k)^{-2/3},
\]
which proves the claim.
\end{proof}

The preceding argument uses only that \(k\mapsto q_k/m\) is non-increasing.
It already gives a lower bound on the critical tail with exponent strictly below \(1\), which
will be enough for the postcritical explosion argument. We now use the exact
triangular system to uncover a stronger structure and improve the
quantitative bounds.

\subsection{Geometric-mixture refinement}

The first layer used only that the tails decrease with their index. We now
identify a stronger property that is preserved by the exact evolution.
A sequence \((M_n)_{n\geq0}\), with \(M_0=1\), is a \emph{Hausdorff
moment sequence} if
\[
        M_n=\int_{[0,1]}z^n\,\rho(dz),\qquad n\geq0,
\]
for a probability measure \(\rho\). Applied to \(M_n=q_{n+1}\), this is
exactly a geometric-mixture representation of the endpoint law.

The proof works with finitely many moments at a time. Their admissible
vectors form a compact convex set. A supporting hyperplane of this set can
be encoded by a polynomial that is nonnegative on \([0,1]\) and vanishes
on the support of a representing measure. We shall check that the tail
vector field points into every such supporting half-space, then pass to
all moments by compactness. We first record the finite-dimensional
invariance criterion needed for this argument.

\begin{lemma}[Nagumo criterion for convex sets]
\label{lem:convex-invariance}
Let \(K\subset\mathbb R^N\) be a non-empty compact convex set and let
\(F:\mathbb R^N\to\mathbb R^N\) be locally Lipschitz. Assume that, for every
\(x\in K\) and every \(a\in\mathbb R^N\) satisfying
\[
        a\cdot(y-x)\leq0\qquad\text{for all }y\in K,
\]
one has \(a\cdot F(x)\leq0\). Then every solution of \(\dot x=F(x)\)
starting in \(K\) remains in \(K\) throughout its interval of existence.
\end{lemma}

\begin{proof}
Let \(P_Kz\) denote the metric projection onto \(K\). For every
\(z\in\mathbb R^N\),
\[
        (z-P_Kz)\cdot(y-P_Kz)\leq0,
        \qquad y\in K,
\]
and \(d_K(z)^2:=\operatorname{dist}(z,K)^2\) is continuously differentiable
with \(\nabla(d_K^2/2)=z-P_Kz\). On a compact neighbourhood of \(K\), let
\(L\) be a Lipschitz constant for \(F\). If \(x(\cdot)\) is a solution and
\(y=P_Kx\), the assumption applied at \(y\) to \(a=x-y\) gives
\((x-y)\cdot F(y)\leq0\). Hence, as long as the solution stays in this
neighbourhood,
\[
 \frac{d}{dt}\frac12d_K(x(t))^2
 =(x-y)\cdot F(x)
 \leq L|x-y|^2=Ld_K(x(t))^2.
\]
If \(x(0)\in K\), Gronwall's inequality gives \(d_K(x(t))=0\). Thus the
solution cannot leave \(K\), and continuation proves the claim.
\end{proof}

\begin{proposition}[Geometric-mixture invariance]
\label{prop:geometric-mixture}
For every \(t\geq0\), there exists a unique probability measure \(\rho_t\)
on \([0,1]\) such that
\begin{equation}
\label{eq:geometric-mixture}
        q_k(t)=\int_{[0,1]}z^{k-1}\,\rho_t(dz),
        \qquad k\geq1.
\end{equation}
Moreover \(t\mapsto\rho_t\) is weakly continuous and
\begin{equation}
\label{eq:atom-one}
        q_\infty(t)=\rho_t(\{1\}).
\end{equation}
Equivalently, conditionally on a mixing parameter \(z<1\), \(X_t\) is
geometric on \(\{1,2,\ldots\}\) with tail \(z^{k-1}\), while \(z=1\)
represents \(X_t=\infty\).
\end{proposition}

\begin{proof}
We prove invariance of each finite-dimensional moment set before
constructing \(\rho_t\). 

\smallskip\noindent
\emph{Step 1: the closed system for finite moment vectors.}
Set \(M_n(t):=q_{n+1}(t)\) for \(n\geq0\), so \(M_0\equiv1\).
The gain--loss form of the tail system reads
\begin{equation}
\label{eq:moment-system}
 M_n'
 =\sum_{j=0}^{n-1}M_jM_{n-1-j}
  -M_n\sum_{j=1}^{n}M_j,
 \qquad n\geq1.
\end{equation}
Fix \(N\geq1\). Write \(\mathcal P([0,1])\) for the probability measures
on \([0,1]\), and define
\[
 \mathcal M_N:=
 \left\{
 \left(\int z\,d\rho,\ldots,\int z^N\,d\rho\right):
 \rho\in\mathcal P([0,1])
 \right\}.
\]
This subset of \(\mathbb R^N\) is convex and compact, being the image of
the weakly compact set \(\mathcal P([0,1])\) under a continuous affine map.
The first \(N\) equations in \eqref{eq:moment-system} form a closed
polynomial system
\(\dot{\mathbf m}=F_N(\mathbf m)\). It remains to show that
\(\mathcal M_N\) is invariant under this system.

\smallskip\noindent
\emph{Step 2: expressing the vector field on polynomials.}
Fix \(\mathbf m\in\mathcal M_N\), and choose any representing measure
\(\rho\). For a polynomial \(g\), put
\begin{equation}
\label{eq:B-rho}
 \mathcal B_\rho g
 :=\iint
 \left[
   \frac{g(x)-g(y)}{x-y}
   +\frac{y}{1-y}\bigl(g(xy)-g(x)\bigr)
 \right]\rho(dx)\rho(dy).
\end{equation}
The first term is extended to \(g'(x)\) on \(x=y\), and the second to
\(-xg'(x)\) on \(y=1\). Both are continuous on \([0,1]^2\); for
polynomial \(g\), these extensions are themselves polynomials. Thus atoms
on the diagonal or at the endpoints cause no integrability issue.

For \(g(z)=z^n\), the two terms are
\[
 \frac{x^n-y^n}{x-y}=\sum_{j=0}^{n-1}x^{n-1-j}y^j,
 \qquad
 \frac{y}{1-y}(x^ny^n-x^n)=-x^n\sum_{j=1}^{n}y^j.
\]
Integrating reproduces the two terms of \eqref{eq:moment-system}.
Consequently, if \(g(z)=g_0+\sum_{n=1}^Ng_nz^n\), then
\begin{equation}
\label{eq:B-vector-field}
        \mathcal B_\rho g
        =\sum_{n=1}^{N}g_n\bigl(F_N(\mathbf m)\bigr)_n.
\end{equation}
In particular, constants are annihilated, and for \(\deg g\leq N\) this
quantity depends only on \(\mathbf m\), not on the chosen representing
measure.

\smallskip\noindent
\emph{Step 3: checking the inward-pointing condition.}
Let \(a=(a_1,\ldots,a_N)\) be an outward normal at \(\mathbf m\), so
\[
        a\cdot(\mathbf v-\mathbf m)\leq0,
        \qquad \mathbf v\in\mathcal M_N.
\]
Define
\[
        p(z):=\sum_{n=1}^{N}a_nz^n,
        \qquad c:=\int p\,d\rho,
        \qquad f(z):=c-p(z).
\]
The moment vector of \(\delta_z\) is in \(\mathcal M_N\) for every
\(z\in[0,1]\), so the supporting inequality says \(p(z)\leq c\).
Thus \(f\geq0\) on \([0,1]\), whereas \(\int f\,d\rho=0\). By
continuity, \(f\) vanishes on \(S:=\operatorname{supp}\rho\). Formula
\eqref{eq:B-vector-field} reduces the required inequality to
\[
        a\cdot F_N(\mathbf m)
        =\mathcal B_\rho p=-\mathcal B_\rho f\leq0.
\]
We therefore check that the integrand defining \(\mathcal B_\rho f\)
is nonnegative on \(S\times S\).

At an interior zero of a nonnegative polynomial the derivative is zero.
At boundary zeros, its sign satisfies
\[
 f'(x)=0\quad(x\in S\cap(0,1)),\qquad
 f'(0)\geq0\quad(0\in S),\qquad
 f'(1)\leq0\quad(1\in S).
\]
For \(x,y\in S\) with \(y<1\), the second term of
\eqref{eq:B-rho} is \(y f(xy)/(1-y)\geq0\), because \(xy\in[0,1]\).
The first term is zero when \(x\ne y\). On the diagonal it is
\(f'(x)\), which is zero for \(0<x<1\) and nonnegative at \(x=0\).

If \(y=1\) and \(x<1\), the first term is zero, and the second is
\(-xf'(x)=0\): an interior support point has zero derivative, while
\(x=0\) supplies a zero factor. Finally, at \((x,y)=(1,1)\), the two
extended terms cancel exactly:
\[
        f'(1)-f'(1)=0.
\]
This cancellation is essential: \(f'(1)\) alone need not be nonnegative.
We have proved \(\mathcal B_\rho f\geq0\), and hence the inward-pointing
condition at every supporting hyperplane. By
Lemma~\ref{lem:convex-invariance}, \(\mathcal M_N\) is invariant.

\smallskip\noindent
\emph{Step 4: constructing and identifying the mixing measure.}
The initial moments \(M_0(0)=1\), \(M_n(0)=0\) for \(n\geq1\), are
those of \(\delta_0\). Invariance therefore implies
\((M_1(t),\ldots,M_N(t))\in\mathcal M_N\) for every \(N\) and \(t\geq0\).
For fixed \(t\), consider
\[
 \mathcal C_N(t):=
 \left\{\rho\in\mathcal P([0,1]):
 \int z^n\,d\rho=M_n(t),\ 1\leq n\leq N\right\}.
\]
These sets are nonempty, weakly compact, and nested as \(N\) increases.
Their intersection is nonempty, and any measure in it satisfies
\eqref{eq:geometric-mixture}. Two such measures agree on all polynomials
and hence, by uniform approximation, on all continuous functions on
\([0,1]\). The representing measure \(\rho_t\) is therefore unique.

If \(t_j\to t\), compactness gives a weakly convergent subsequence of any
subsequence of \((\rho_{t_j})\). Continuity of each \(M_n\) shows that
its limit represents the moments at \(t\), so uniqueness identifies it
as \(\rho_t\). This proves weak continuity.

Finally, for each fixed \(t\), dominated convergence gives
\[
        q_\infty(t)
        =\lim_{k\to\infty}\int z^{k-1}\,\rho_t(dz)
        =\rho_t(\{1\}).
\]
For finite \(k\), subtracting consecutive tails gives
\(\mathbb P(X_t=k)=\int_{[0,1)}(1-z)z^{k-1}\,\rho_t(dz)\).
This proves both the atom identity and the stated geometric-mixture
interpretation.
\end{proof}

\begin{lemma}[Maximum of two geometric mixtures]
\label{lem:geometric-maximum}
Let \(K,K'\) be independent and identically distributed random variables on
\(\N\). Assume that their common law is a mixture of geometric laws on
\(\{1,2,\ldots\}\) and that \(\mathbb E[K]<\infty\). Then
\begin{equation}
\label{eq:geometric-maximum}
        \mathbb E[K\vee K']
        \geq
        \frac32\mathbb E[K]-\frac14-\frac14\mathbb P(K=1).
\end{equation}
\end{lemma}

\begin{proof}
Condition on the two mixing parameters, parametrizing a geometric law by its
mean. Thus, conditionally on \(a,b\geq1\), let \(G_a,G_b\) be independent
geometric random variables on \(\{1,2,\ldots\}\) with means \(a,b\). Their
tail parameters are \(1-a^{-1}\) and \(1-b^{-1}\), so
\[
        \mathbb E[G_a\wedge G_b]=\frac{ab}{a+b-1}.
\]
The identity
\[
 \frac{ab}{a+b-1}
 =\frac{a+b+1}{4}
  +\frac{1-(a-b)^2}{4(a+b-1)}
\]
and \(a,b\geq1\) imply
\[
 \frac{ab}{a+b-1}
 \leq\frac{a+b+1}{4}
      +\frac18\left(\frac1a+\frac1b\right),
\]
because \(2ab\leq(a+b)(a+b-1)\), equivalently
\(a(a-1)+b(b-1)\geq0\). Since
\(G_a\vee G_b+G_a\wedge G_b=G_a+G_b\),
\[
 \mathbb E[G_a\vee G_b]
 \geq\frac34(a+b)-\frac14
      -\frac18\left(\frac1a+\frac1b\right).
\]
Averaging over two independent copies of the mixing parameter gives
\[
 \mathbb E[K\vee K']
 \geq\frac32\mathbb E[K]-\frac14
      -\frac14\mathbb E[a^{-1}].
\]
Conditionally on \(a\), the probability that the geometric variable equals
\(1\) is \(a^{-1}\). Hence \(\mathbb E[a^{-1}]=\mathbb P(K=1)\), proving
\eqref{eq:geometric-maximum}.
\end{proof}

\begin{proposition}[Improved compensation estimate]
\label{prop:moment-comparison}
For every \(t<t_c\),
\begin{equation}
\label{eq:moment-comparison}
        A(t)\leq m(t)^{5/2}.
\end{equation}
Consequently,
\begin{equation}
\label{eq:improved-second-moment}
        \mathbb E[X_t^2]\leq2m(t)^{5/2}.
\end{equation}
\end{proposition}

\begin{proof}
Fix \(t<t_c\) and suppress \(t\) from the notation. By
Proposition~\ref{prop:geometric-mixture},
\(q_k=\int z^{k-1}\rho(dz)\). Since \(m<\infty\), Tonelli's theorem gives
\(m=\int(1-z)^{-1}\rho(dz)<\infty\), so \(\rho(\{1\})=0\). Hence
\[
        \widehat\rho(dz):=\frac{\rho(dz)}{m(1-z)}
\]
is a probability measure and
\begin{equation}
\label{eq:normalized-tail-geometric}
        \frac{q_k}{m}
        =\int_{[0,1)}(1-z)z^{k-1}\,\widehat\rho(dz).
\end{equation}
Thus the random variable \(K\) from
Lemma~\ref{lem:weighted-tail-identity} is a mixture of geometric laws, with
\[
        \mathbb E[K]=\frac{A}{m},
        \qquad \mathbb P(K=1)=\frac1m.
\]
Returning to the exact identity \eqref{eq:A-prime-maximum} and applying
Lemma~\ref{lem:geometric-maximum},
\[
        A'\leq\frac54mA+A+\frac18(m^2+m)
        -\frac12\sum_{k\geq1}kq_k^2.
\]
Since \(m^2\leq2A-m\), one has \(m^2+m\leq2A\), and therefore
\[
        A'\leq\frac54(m+1)A-\frac12\sum_{k\geq1}kq_k^2
        \leq\frac54(m+1)A.
\]
Using \eqref{eq:riccati-lower},
\[
        \frac{A'}A\leq\frac54(m+1)
        \leq\frac52\frac{m'}m.
\]
As \(A(0)=m(0)=1\), integration proves \eqref{eq:moment-comparison}. The
second assertion follows from \(\mathbb E[X_t^2]=2A(t)-m(t)\).
\end{proof}

\begin{proposition}[Lower bound on the critical tail]
\label{prop:critical-lower-tail}
There exists \(c>0\) such that
\begin{equation}
\label{eq:critical-lower-tail}
        q_k(t_c)=\mathbb P(X_{t_c}\geq k)\geq ck^{-1/2},
        \qquad k\geq1.
\end{equation}
One may take \(c=(8\sqrt2)^{-1}\). In particular,
\(\mathbb E[X_{t_c}]=\infty\).
\end{proposition}

\begin{proof}
By Proposition~\ref{prop:moment-comparison},
\(\mathbb E[X_t^2]\leq2m(t)^{5/2}\) for \(t<t_c\). Paley--Zygmund gives
\[
 \mathbb P\!\left(X_t\geq\frac12m(t)\right)
 \geq\frac14\frac{m(t)^2}{\mathbb E[X_t^2]}
 \geq\frac18m(t)^{-1/2}.
\]
Since \(m\) is continuous on \([0,t_c)\) and tends to infinity at
\(t_c\), for every \(k\geq1\) there exists \(t_k<t_c\) with
\(m(t_k)=2k\). Monotonicity in the level yields
\[
        q_k(t_c)\geq q_k(t_k)
        \geq\frac1{8\sqrt2}\,k^{-1/2}.
\]
Summing in \(k\) gives \(\mathbb E[X_{t_c}]=\infty\).
\end{proof}

\begin{corollary}[Critical fractional moments]
\label{cor:critical-fractional-moments}
For every \(p\geq1/3\),
\begin{equation}
\label{eq:critical-fractional-moments}
        \mathbb E[X_{t_c}^p]=\infty.
\end{equation}
Consequently, if the logarithmic critical-tail exponent
\[
        \alpha:=\lim_{k\to\infty}
        \frac{-\log q_k(t_c)}{\log k}
\]
exists, then \(\alpha\leq1/3\).
\end{corollary}

\begin{proof}
We prove the endpoint case \(p=1/3\) by a truncation argument. Suppose,
for a contradiction, that \(\mathbb E[X_{t_c}^{1/3}]<\infty\). For
\(R\geq1\), set
\[
        H_R:=\mathbb E\!\left[
          X_{t_c}^{1/3}\mathbf 1_{\{X_{t_c}>R\}}
        \right].
\]
Then \(H_R\to0\) as \(R\to\infty\). Under the monotone coupling,
\(X_t\leq X_{t_c}\) for \(t<t_c\), so
\[
 \mathbb E\!\left[X_t^{1/3}\mathbf 1_{\{X_t>R\}}\right]\leq H_R.
\]
For every \(t<t_c\), H\"older's inequality with conjugate exponents
\(5/3\) and \(5/2\), followed by \eqref{eq:improved-second-moment}, gives
\[
\begin{aligned}
 m(t)
 &\leq R+\mathbb E\!\left[X_t\mathbf 1_{\{X_t>R\}}\right] \\
 &\leq R+
 \mathbb E\!\left[X_t^{1/3}\mathbf 1_{\{X_t>R\}}\right]^{3/5}
 \mathbb E\!\left[X_t^2\mathbf 1_{\{X_t>R\}}\right]^{2/5} \\
 &\leq R+H_R^{3/5}\bigl(2m(t)^{5/2}\bigr)^{2/5} \\
 &=R+2^{2/5}H_R^{3/5}m(t).
\end{aligned}
\]
Dividing by \(m(t)\) and letting \(t\uparrow t_c\), using
\(m(t)\to\infty\), yields
\[
        1\leq2^{2/5}H_R^{3/5},\qquad R\geq1.
\]
Letting \(R\to\infty\) gives a contradiction. Hence
\(\mathbb E[X_{t_c}^{1/3}]=\infty\). Since \(X_{t_c}\geq1\), it follows
that \(\mathbb E[X_{t_c}^{p}]=\infty\) for every \(p\geq1/3\).

If \(\alpha>1/3\), choose \(p\in(1/3,\min\{\alpha,1\})\) and then
\(\beta\in(p,\alpha)\). By the definition of \(\alpha\),
\(q_k(t_c)\leq k^{-\beta}\) for all sufficiently large \(k\). Since
\[
 \mathbb E[X_{t_c}^p]
 =\sum_{k\geq1}\bigl(k^p-(k-1)^p\bigr)q_k(t_c)
\]
and \(k^p-(k-1)^p=O(k^{p-1})\), this would make the \(p\)-th moment finite,
a contradiction. Thus \(\alpha\leq1/3\).
\end{proof}

\subsection{Critical finiteness}
\leavevmode\par
The quantitative estimates above leave open whether some mass has already
escaped to infinity at the critical level. The triangular structure rules
this out.

\begin{proposition}[Almost sure finiteness at criticality]
\label{prop:critical-no-infinity}
At the susceptibility threshold,
\[
        q_\infty(t_c)
        =
        \mathbb P(X_{t_c}=\infty)
        =0.
\]
Thus \(X_{t_c}<\infty\) almost surely.
\end{proposition}

\begin{proof}
Set
\[
        a:=q_\infty(t_c).
\]
Since \eqref{eq:riccati-lower} implies \(m'\geq m^2/2\) and
\(m(t)\to\infty\) as \(t\uparrow t_c\), integration of
\((1/m)'\leq-1/2\) gives
\begin{equation}
\label{eq:critical-m-upper}
        m(t)\leq\frac{2}{t_c-t},
        \qquad t<t_c.
\end{equation}

Fix \(C>0\), and for all sufficiently large \(n\) set
\[
        t_n:=t_c-\frac{C}{n}.
\]
Since \(k\mapsto q_k(t_n)\) is non-increasing,
\[
        nq_n(t_n)
        \leq m(t_n)
        \leq\frac{2n}{C},
\]
and hence
\begin{equation}
\label{eq:critical-diagonal-upper}
        q_n(t_n)\leq\frac2C.
\end{equation}
We claim that
\begin{equation}
\label{eq:critical-diagonal-limit}
        q_n(t_n)\longrightarrow a.
\end{equation}

The upper bound is immediate from monotonicity in the level:
\(q_n(t_n)\leq q_n(t_c)\to a\). For the reverse bound, let
\(s\in[t_n,t_c]\). For \(1\leq u\leq n-1\), monotonicity in the spatial
index and in the level gives
\[
        q_n(s)
        \leq q_{n-u}(s)
        \leq q_{n-u}(t_c).
\]
Using \eqref{eq:tail-system}, \(q_u(s)\leq1\), and
\(q_n(s)(1-q_n(s))\leq1\), we obtain
\[
\begin{aligned}
 q_n'(s)
 &\leq
 \sum_{u=1}^{n-1}
 \bigl(q_{n-u}(t_c)-q_n(s)\bigr)+1 \\
 &=(n-1)\bigl(b_n-q_n(s)\bigr),
\end{aligned}
\]
where
\[
        b_n:=
        \frac{1+\sum_{j=1}^{n-1}q_j(t_c)}{n-1}.
\]
Gronwall's inequality over \([t_n,t_c]\) yields
\[
 q_n(t_c)
 \leq
 e^{-(n-1)C/n}q_n(t_n)
 +b_n\bigl(1-e^{-(n-1)C/n}\bigr).
\]
Now \(q_n(t_c)\to a\), while \(b_n\to a\) by Cesàro convergence.
Rearranging the preceding inequality gives
\[
        \liminf_{n\to\infty}q_n(t_n)\geq a.
\]
Together with the already proved upper bound, this establishes
\eqref{eq:critical-diagonal-limit}.

Finally, \eqref{eq:critical-diagonal-upper} and
\eqref{eq:critical-diagonal-limit} imply \(a\leq2/C\). Since \(C>0\)
is arbitrary, \(a=0\).
\end{proof}

Combining Proposition~\ref{prop:critical-lower-tail},
Corollary~\ref{cor:critical-fractional-moments}, and
Proposition~\ref{prop:critical-no-infinity}, the critical avalanche is finite
almost surely but has infinite mean. The elementary bound
\eqref{eq:first-critical-lower-tail} already has exponent below \(1\), whereas
the geometric-mixture refinement gives the stronger pointwise and
fractional-moment information. The next section shows that either bound is
heavy enough to force explosion at every level strictly above \(t_c\).

\section{Postcritical explosion}
\label{sec:postcritical}

We now derive postcritical explosion from the critical tail estimates.
The elementary estimate proves sharpness, while the refined estimate
improves the quantitative lower bound. We first isolate the mechanism in
a time-homogeneous process with a frozen graft law.

Let \(Y\) be a random variable taking values in \(\mathbb N\), and write
\[
        Q_k:=\mathbb P(Y\geq k),\qquad k\geq1.
\]
Given an initial state \(Z_0\in\mathbb N\), define the increasing process
\(Z\) by independent driving variables: let \((E_j)_{j\geq0}\) be unit exponentials,
\((V_j)_{j\geq0}\) uniforms on \((0,1)\), and \((Y_j)_{j\geq0}\) copies
of \(Y\), with all variables mutually independent and independent of
\(Z_0\). Set \(T_0=0\), \(z_0=Z_0\), and, recursively,
\[
 \begin{split}
        T_{j+1}&=T_j+\frac{E_j}{z_j+1},\\
        U_j&=\lfloor(z_j+1)V_j\rfloor,\qquad
        z_{j+1}=\max\{z_j,U_j+Y_j\}.
 \end{split}
\]
Put \(Z_s=z_j\) on \([T_j,T_{j+1})\), and let
\(\zeta=\lim_{j\to\infty}T_j\). At a finite accumulation time \(\zeta\),
send the process to \(\infty\), which is absorbing. Thus, conditionally on
the current finite state \(x\), the waiting time to the next proposed
update is exponential of rate \(x+1\), its origin is uniform on
\(\{0,\ldots,x\}\), and its graft has law \(Y\). The actual waiting
times depend on the successive states; it is the driving variables that
are independent. Proposed updates that leave the endpoint unchanged are
retained in this construction.

\begin{proposition}[Frozen heavy-tail explosion]
\label{prop:frozen-heavy-tail}
Suppose that
\[
        Q_k\geq ck^{-\alpha},\qquad k\geq1,
\]
for some \(c>0\) and \(0<\alpha<1\). If \(Z_0\) has the law of \(Y\), then
\[
        \mathbb P(\zeta\leq s)>0,\qquad s>0.
\]
More precisely, there are constants \(c_0>0\) and \(s_0>0\), depending
only on \(c\) and \(\alpha\), such that
\begin{equation}
\label{eq:frozen-quantitative}
        \mathbb P(\zeta\leq s)
        \geq c_0s^{\alpha/(1-\alpha)},
        \qquad 0<s\leq s_0.
\end{equation}
\end{proposition}

\begin{proof}
Fix \(x\geq1\), and let
\(\tau_{2x}:=\inf\{s\geq0:Z_s\geq2x\}\). While
\(x\leq Z_s<2x\), proposed updates whose graft is at least \(2x\) occur
at rate at least \(xQ_{2x}\), and each such update reaches level \(2x\).
Consequently, uniformly in \(z\in[x,2x)\),
\begin{equation}
\label{eq:dyadic-crossing}
        \mathbb E_z[\tau_{2x}]
        \leq\frac{1}{xQ_{2x}}
        \leq Cx^{\alpha-1}.
\end{equation}

Fix \(K\geq1\) and \(z\geq K\), and choose \(j\geq0\) such that
\(2^jK\leq z<2^{j+1}K\). For \(i\geq j+1\), let
\(\sigma_i:=\inf\{s\geq0:Z_s\geq2^iK\}\). On every bounded set of
states the proposed-update rates are bounded, so update times cannot
accumulate while \(Z\) remains bounded. Consequently
\(\zeta=\lim_{i\to\infty}\sigma_i\). At time \(\sigma_i\), either the
process has already crossed \(2^{i+1}K\), or its state lies in
\([2^iK,2^{i+1}K)\); in the latter case \eqref{eq:dyadic-crossing} and the
strong Markov property give
\[
 \mathbb E_z[\sigma_{i+1}-\sigma_i\mid\mathcal F_{\sigma_i}]
 \leq C(2^iK)^{\alpha-1}.
\]
The analogous estimate for the first crossing from \(z\) to
\(2^{j+1}K\), followed by monotone convergence, yields
\[
        \mathbb E_z[\zeta]
        \leq C\sum_{i\geq j}(2^iK)^{\alpha-1}
        \leq C'K^{\alpha-1}.
\]
The series converges because \(\alpha<1\). Given \(s>0\), choose \(K\)
large enough that \(C'K^{\alpha-1}\leq s/2\). Markov's inequality then
implies
\[
        \inf_{z\geq K}\mathbb P_z(\zeta\leq s)\geq\frac12.
\]
Since \(Z_0\) has the law of \(Y\),
\[
        \mathbb P(\zeta\leq s)
        \geq
        Q_K\inf_{z\geq K}\mathbb P_z(\zeta\leq s)
        \geq\frac12Q_K>0.
\]

For the quantitative bound, keep the dependence of \(K\) on \(s\).
The constants \(C\) and \(C'\) above depend only on \(c\) and
\(\alpha\). Set \(s_0=\min\{1,2C'\}\) and, for \(0<s\leq s_0\), take
\[
        a_s:=\left(\frac{2C'}s\right)^{1/(1-\alpha)},
        \qquad K=\lceil a_s\rceil.
\]
Then \(a_s\geq1\), \(a_s\leq K\leq2a_s\), and
\(C'K^{\alpha-1}\leq s/2\). The same Markov bound therefore yields
\[
 \mathbb P(\zeta\leq s)
 \geq\frac c2 K^{-\alpha}
 \geq\frac{c}{2^{1+\alpha}}
       (2C')^{-\alpha/(1-\alpha)}s^{\alpha/(1-\alpha)}.
\]
This proves \eqref{eq:frozen-quantitative}.
\end{proof}

We apply the proposition with the critical avalanche law. Set
\[
        Q_k:=q_k(t_c),\qquad k\geq1.
\]
By Proposition~\ref{prop:critical-no-infinity}, \(Q_k\downarrow0\), so
there is a finite \(\{1,2,\ldots\}\)-valued random variable \(Y\) with
\(\mathbb P(Y\geq k)=Q_k\). Let \(Z\) be the corresponding frozen
process, started with the law of \(Y\), and set
\[
        r_k(s):=\mathbb P(Z_s\geq k).
\]

\begin{proposition}[Postcritical domination]
\label{prop:postcritical-domination}
For every \(k\geq1\) and \(s\geq0\),
\[
        q_k(t_c+s)\geq r_k(s).
\]
Consequently,
\[
        q_\infty(t_c+s)\geq\mathbb P(\zeta\leq s).
\]
\end{proposition}

\begin{proof}
The same crossing computation as in
Proposition~\ref{prop:tail-system}, now with the graft law frozen at
\(Q\), gives \(r_1\equiv1\) and, for \(k\geq2\),
\begin{equation}
\label{eq:frozen-tail-system}
        r_k'
        =
        \sum_{u=1}^{k-1}Q_u(r_{k-u}-r_k)
        +Q_k(1-r_k),
        \qquad r_k(0)=Q_k.
\end{equation}
Set \(\widetilde q_k(s):=q_k(t_c+s)\). By monotonicity in the level,
\(\widetilde q_u(s)\geq Q_u\). Since the tail sequence is decreasing,
\(\widetilde q_{k-u}(s)-\widetilde q_k(s)\geq0\), while
\(1-\widetilde q_k(s)\geq0\). Hence \eqref{eq:tail-system} yields
\[
        \widetilde q_k'
        \geq
        \sum_{u=1}^{k-1}Q_u
        (\widetilde q_{k-u}-\widetilde q_k)
        +Q_k(1-\widetilde q_k).
\]
We prove \(\widetilde q_k\geq r_k\) by induction on \(k\). The case
\(k=1\) is immediate. If the assertion holds below level \(k\), put
\(\Delta_k:=\widetilde q_k-r_k\). Subtracting
\eqref{eq:frozen-tail-system} gives
\[
        \Delta_k'
        +\left(\sum_{u=1}^{k}Q_u\right)\Delta_k
        \geq
        \sum_{u=1}^{k-1}Q_u\Delta_{k-u}
        \geq0,
        \qquad \Delta_k(0)=0.
\]
An integrating factor gives \(\Delta_k\geq0\), completing the induction.
Letting \(k\to\infty\), and recalling that \(Z_s=\infty\) exactly when
\(\zeta\leq s\), proves the second assertion.
\end{proof}

We can now distinguish the roles of the two critical estimates.
Corollary~\ref{prop:first-critical-lower-tail} gives
\(Q_k\geq ck^{-2/3}\). Applying
Propositions~\ref{prop:frozen-heavy-tail} and~\ref{prop:postcritical-domination}
with \(\alpha=2/3\) already yields
\[
        q_\infty(t_c+s)>0,\qquad s>0,
\]
and even a lower bound proportional to \(s^2\) for small \(s\). Together
with Proposition~\ref{prop:critical-no-infinity} and finite susceptibility
below \(t_c\), this proves
\[
        t_c=\inf\{t\geq0:\mathbb P(X_t=\infty)>0\}.
\]
Thus the geometric-mixture argument is not needed for sharpness itself.
It strengthens the quantitative conclusion as follows.

\begin{corollary}[Quantitative postcritical bound]
\label{cor:quantitative-postcritical}
There exist constants \(c_*>0\) and \(s_0>0\) such that
\begin{equation}
\label{eq:postcritical-linear}
        \mathbb P(X_{t_c+s}=\infty)\geq c_*s,
        \qquad 0<s\leq s_0.
\end{equation}
\end{corollary}

\begin{proof}
Proposition~\ref{prop:critical-lower-tail} gives
\(Q_k\geq ck^{-1/2}\). The quantitative estimate
\eqref{eq:frozen-quantitative}, with \(\alpha=1/2\), gives
\(\mathbb P(\zeta\leq s)\geq c_*s\) for \(0<s\leq s_0\).
Proposition~\ref{prop:postcritical-domination} transfers this bound to
\(q_\infty(t_c+s)\).
\end{proof}

The bound \eqref{eq:postcritical-linear} is one-sided: it does not assert
a linear asymptotic for the infinite-avalanche probability. Combined with
Proposition~\ref{prop:susceptibility},
Proposition~\ref{prop:exponential-moments}, and the critical results of
Section~\ref{sec:critical-tail}, it completes the proof of
Theorem~\ref{thm:main}.

\appendix
\section{Use of AI}
\label{appendix-AI}

We began the investigation of the Bak--Sneppen model two years ago by seeking soft,
model-independent PDE arguments for sharpness. After several unsuccessful
attempts, described in the appendix of \cite{HenardZd}, we turned to
the one-dimensional one-sided, or anisotropic, model. We expected its
dynamics to be more tractable than those of the isotropic model, and
hoped that sharpness could subsequently be transferred from one to
the other. For the anisotropic model, we first derived the triangular system
governing the tails of the avalanche range.

In July 2026, taking this triangular system as input, ChatGPT 5.6 Sol
proposed the compensation argument that led to our first proof of
sharpness for the anisotropic model, and which also corresponds to the first argument presented here, leading to Proposition \ref{prop:first-moment-comparison}. We then observed that stochastic
domination allowed us to transfer sharpness to the isotropic model,
as presented in the first version of this paper.

For the present revision, we asked GPT-6 Astra to improve the
critical tail exponent.
The model proposed a refinement based on the preservation
of geometric mixtures under the evolution, Proposition \ref{prop:geometric-mixture}, yielding the sharper
exponent.

Claude did the simulation leading to the  picture.

\bibliographystyle{amsplain}
\bibliography{references}

\end{document}